\documentclass[a4paper,12pt]{amsart}
\usepackage{amsfonts}
\usepackage{amsmath}
\usepackage{amssymb}
\usepackage[a4paper]{geometry}
\usepackage{mathrsfs}
\usepackage{hyperref}
\def\G{G}

\def\C{\mathbb C}

\DeclareMathOperator{\trace}{trace}
\def\d{\mathrm d}

\DeclareMathOperator{\rank}{rank}
\def\Gh{{\widehat{\G}}}
\def\SU2{{{\rm SU}(2)}}

\def\R{{\mathbb R}}
\def\D{{\mathbb D}}

\def\N{{\mathbb N}}
\def\Z{{\mathbb Z}}
\def\C{{\mathbb C}}
\def\S3{{{\mathbb S}^3}}
\def\Rn{{{\mathbb R}^n}}
\def\Tn{{{\mathbb T}^n}}
\def\Zn{{{\mathbb Z}^n}}

\def\p#1{{\left({#1}\right)}}

\def\jp#1{{\left\langle{#1}\right\rangle}}

\DeclareMathOperator{\diff}{diff}
\DeclareMathOperator{\sign}{sign}
\DeclareMathOperator{\op}{Op}

\def\d{\mathrm d}
\def\e{\mathrm e}
\def\i{\mathrm i}
\let\Re\relax

\DeclareMathOperator\Re{Re}

\def\rhodiff{\mbox{$\triangle\!\!\!\!\ast\,$}}

\def\HS{{\mathtt{HS}}}
\newtheorem{thm}{Theorem}[section]
\newtheorem{lem}[thm]{Lemma}
\newtheorem{cor}[thm]{Corollary}
\newtheorem*{crit}{Criterion}
\theoremstyle{remark}
\newtheorem{rem}[thm]{Remark}
\newtheorem{expl}[thm]{Example}
\numberwithin{equation}{section}
\begin{document}
\title[Multipliers on compact Lie groups]
{$L^p$ Fourier multipliers on compact Lie groups}
\author[Michael Ruzhansky]{Michael Ruzhansky}
\address{
  Michael Ruzhansky:
  \endgraf
  Department of Mathematics
  \endgraf
  Imperial College London
  \endgraf
  180 Queen's Gate, London SW7 2AZ 
  \endgraf
  United Kingdom
  \endgraf
  {\it E-mail address} {\rm m.ruzhansky@imperial.ac.uk}
  }
\author[Jens Wirth]{Jens Wirth}
\address{
  Jens Wirth:
  \endgraf
  Institut f\"ur Analysis, Dynamik und Modellierung
  \endgraf
  Universit\"at Stuttgart
  \endgraf
  Pfaffenwaldring 57, 70569 Stuttgart
  \endgraf
  Germany
  \endgraf
  {\it E-mail address} {\rm jens.wirth@iadm.uni-stuttgart.de}
  \endgraf
  }
\thanks{The first
 author was supported by the EPSRC
 Leadership Fellowship EP/G007233/1. The second author
 was supported by the EPSRC grant EP/E062873/1 and by 
  DAAD for visits to Imperial College London
  in December 2010 and February 2011.
 }
 
\date{\today}

 
\begin{abstract}
In this paper we prove $L^p$ Fourier multiplier theorems for
invariant and also non-invariant operators on compact Lie groups
in the spirit of the well-known H\"ormander-Mikhlin theorem 
on $\mathbb R^n$ and its variants on tori $\mathbb T^n$. 
We also give applications to a-priori estimates for non-hypoelliptic
operators. Already in the case of tori we get an interesting refinement
of the classical multiplier theorem.
\end{abstract}
\maketitle
 
\section{Introduction}

In this paper we prove $L^p$ multiplier theorems for
invariant and then also 
for non-invariant operators on compact Lie groups.
We are primarily interested in Fourier multipliers rather than in spectral multipliers.

The topic has been attracting intensive research for a long time.
There is extensive literature providing criteria for central
multipliers, see e.g. N. Weiss  \cite{W72},
Coifman and G. Weiss \cite{CW70}, 
Stein \cite{Stein70}, Cowling \cite{Cowling}, 
Alexopoulos \cite{Alex}, to mention only very few. 
There are also results for functions of the sub-Laplacian,
for example on $\SU2$, see Cowling and Sikora \cite{CS}.

The topic of the $L^p$-bounded multipliers
has been extensively researched on symmetric
spaces of noncompact type for multipliers corresponding
to convolutions with distributions which are
bi-invariant with respect to the subgroup, see e.g. 
Anker \cite{Anker} and references therein.
However, general results on compact Lie groups are
surprisingly elusive. 
For the case of the group 
$\SU2$ a characterisation for operators
leading to Calderon--Zygmund kernels in terms of certain symbols
 was given by 
 Coifman and G. Weiss in \cite{CW70} based on a criterion
 for Calderon--Zygmund 
operators from \cite{CdG70} (see also \cite{CWbook}). 
The proofs and formulations, however, rely on explicit formulae for
representations and for the Clebsch--Gordan 
coefficients available
on $\SU2$ and are not extendable to other groups.
In general, in the case when we do not deal with functions
of a fixed operator, it is even unclear in which
terms to formulate criteria for the $L^p$-boundedness.

In this paper we prove a general result 
for arbitrary compact Lie groups $G$. 
This becomes possible based on the tools 
initiated and developed by the first author and V. Turunen in 
\cite{RT-su2} and \cite{RTbook}, in particular the 
development of
the matrix valued symbols and the corresponding 
quantization relating operators and their symbolic
calculus with
the representation theory of the group.
In view of the results in \cite{RT-su2, RTW10}, 
pseudo-differential operators in H\"ormander classes 
$\Psi^m(\G)$ can
be characterised in terms of decay conditions imposed
on the matrix valued symbols using natural difference
operators acting on the unitary dual $\Gh$. 
From this point of view Theorem \ref{thm:main1}
provides a Mikhlin type multiplier theorem which
reduces the assumptions on the symbol ensuring the
$L^p$-boundedness of the operator.
In Theorem \ref{thm:main2} we give a refinement of this
describing precisely the difference operators that can
be used for making assumptions on the symbol. For example,
if $G$ is semi-simple, only those associated to the
root system suffice, which appears natural in the context.

We give several applications of the obtained result.
Thus, in Corollary \ref{cor2} we give a criterion for
the $L^p$-boundedness for a class of operators with 
symbols in the class 
$\mathscr S^0_\rho(\G)$ of type $\rho\in[0,1]$.
Such operators appear e.g. with $\rho=\frac12$ 
as parametrices for the sub-Laplacian
or for the ``heat'' operator,
see Example \ref{EX:subL}, or with $\rho=0$
for inverses of operators
$X+c$, with $X\in {\mathfrak g}$ and $c\in\C$,
see Corollary \ref{COR:vfs} on general $G$ and
Example \ref{EX:vfs} on $\SU2$ and $\S3$. We note that
although operators $X+c$ are not locally hypoelliptic,
we still get a-priori $L^p$-estimates for them as a consequence of
our result. 

We illustrate Theorem \ref{thm:main2} in 
Remark \ref{REM:tori} in
the special case
of the tori $\Tn$. In different versions of 
multiplier theorems on $\Tn$, one usually
expects to impose conditions on differences of order
$[\frac{n}{2}]+1$ applied to the symbol.
In Remark \ref{REM:tori} we show that e.g. on
${\mathbb T}^2$ or ${\mathbb T}^3$, it is
enough to make an assumption on only one second
order difference of a special form applied to
the symbol. In particular, this improves by now classical
theorems on $L^p$-multipliers requiring $n$ differences, see e.g.
Nikolskii \cite[Section 1.5.3]{N}.

In Theorem \ref{thm:main3} we give an application to 
the $L^p$-estimates for
general
operators from $C^\infty(\G)$ to ${\mathcal D}'(\G)$, not
necessarily invariant. This result is also a relaxation
of the symbolic assumptions on the operator compared to 
those in the pseudo-differential classes. In Theorem \ref{thm:main3} 
we give a condition for symbols based on the $(1,0)$-type behaviour.
Since the number of imposed conditions is finite, it can be extended
further to $(\rho,\delta)$-type conditions similarly to the case of 
multipliers in Section \ref{SEC:general}. In general, symbol classes
of type $(\rho,\delta)$ for matrix symbols on compact Lie
groups were introduced in \cite{RTW10}. These symbols also
satisfy a suitable version of the functional calculus, see
the authors' paper \cite{RW14}.

In \cite{ANR}, Fourier multiplier theorems have been recently obtained for operators to be bounded from
$L^p$ to $L^q$ for $1<p\leq 2\leq q<\infty$ in the setting of the compact Lie group
SU(2). However, those results are different in nature as they explore only the decay
rate of symbols rather than the much more subtle behaviour expressed in terms of 
difference operators in this paper.

The paper is organised as follows. In Section \ref{SEC:Mults}
we formulate the results with several application and
give a number of examples. In Section \ref{SEC:proofs}
we introduce the necessary techniques and
prove the results. In Section \ref{SEC:central} we
briefly discuss central multipliers and the meaning of 
the difference operator $\rhodiff$ in this case.
Finally, in Section \ref{SEC:general} we prove
corollaries for operators with symbols in $\mathscr S^0_\rho(\G)$
and for non-invariant operators.

Some of the results of this paper have been announced in \cite{RW13} without proof. 

\section{Multiplier theorems on compact Lie groups}
\label{SEC:Mults}

Let $\G$ be a compact Lie group with identity $1$ and the 
unitary dual $\Gh$.
The following considerations are based on the group Fourier transform
\begin{equation}\label{eq:FI}
{\mathscr F}\phi = \widehat \phi(\xi) = 
\int_\G \phi(g) \xi(g)^* \d g,\qquad
  \phi(g) = \sum_{[\xi]\in\widehat\G} d_\xi 
  \trace(\xi(g) \widehat \phi(\xi) ) = 
  {\mathscr F}^{-1} [\widehat\phi]  
\end{equation}
defined in terms of equivalence classes 
$[\xi]$ of irreducible unitary 
representations $\xi : \G \to \mathrm U(d_\xi)$ 
of dimension (degree) $d_\xi$. 
The Peter--Weyl theorem on $\G$ implies in 
particular that this pair of 
transforms is inverse to each other and that the Plancherel identity
\begin{equation}\label{eq:Plancherel}
   \|\phi\|_2^2 = \sum_{[\xi]\in\Gh} d_\xi \|\widehat\phi(\xi)\|_{\HS}^2
   =: \|\widehat{\phi}\|_{\ell^2(\Gh)}^2
\end{equation}
holds true for all $\phi\in L^2(\G)$. Here $$\|\widehat
\phi(\xi)\|_{\HS}^2 = \trace(\widehat \phi(\xi)\widehat\phi(\xi)^*)$$
denotes the Hilbert--Schmidt (Frobenius) norm of matrices. 
The Fourier inversion 
statement \eqref{eq:FI} is valid for all 
$\phi\in\mathcal D'(\G)$ and
the Fourier series converges in $C^\infty(\G)$ provided $\phi$ is
smooth. It is further convenient to denote
\begin{equation*}
  \langle\xi\rangle = \max\{1, \lambda_\xi\},
\end{equation*}
where $\lambda_\xi^2$ is the eigenvalue of the Casimir element 
(positive 
Laplace-Beltrami operator) acting on the matrix coefficients 
associated to the representation $\xi$. The Sobolev spaces 
can be characterised by Fourier coefficients as
\begin{equation*}
 \phi\in H^s(\G) \quad\Longleftrightarrow \quad 
 \langle\xi\rangle^s \widehat\phi(\xi) \in \ell^2(\widehat\G), 
\end{equation*}
where $\ell^2(\widehat\G)$ is defined as the space of matrix-valued
sequences such that the sum on the right-hand side of 
\eqref{eq:Plancherel} is finite.

For an arbitrary continuous 
linear operator $A:C^\infty(\G) \to \mathcal D'(\G)$ we denote its
Schwartz kernel as $K_A\in \mathcal D'(\G\times\G)$ 
and by a change of
variables we associate the right-convolution kernel
 $$R_A(g_1,g_2) = K_A(g_1, g_1^{-1}g_2).$$ 
 Thus, at least formally, we write
\begin{equation*}
   A \phi(g_1) = \int_\G K_A(g_1,g_2)\phi(g_2)\d g_2 = 
   \int_\G \phi(g_2) R_A(g_1, g_2^{-1}g_1) \d g_2 = \phi * R_A(g_1,\cdot).
\end{equation*}
Following the analysis
in \cite{RTbook} we denote the partial 
Fourier transform of the right-convolution kernel with respect 
to the second variable as symbol of the operator, 
\begin{equation}\label{EQ:RS}
   \sigma_A(g, \xi)  := \widehat R_A(g,\xi) = 
   \int_\G R_A(g,g') \xi(g')^*\d g' \in 
   \mathcal D'(\G)\widehat\otimes_\pi \Sigma(\widehat\G),
\end{equation}
which is a distribution taking values in the set of moderate 
sequences of matrices
\begin{equation*}
   \Sigma(\widehat\G) = \{ \sigma : \xi \mapsto \sigma(\xi) 
   \in \mathbb C^{d_\xi\times d_\xi} : 
   \| \sigma(\xi) \|_{\rm op} \lesssim \langle 
   \xi\rangle^N \text{ for some $N$}\}.
\end{equation*}

Here we are concerned with left-invariant operators, which 
means that $A\circ T_g = T_g \circ A$ 
for all the left-translations $T_g :
\phi\mapsto \phi(g^{-1} \cdot)$. 
This implies that the kernel $K_A$ satisfies the invariance
 $$K_A(g_1,g_2) = K_A(g^{-1}g_1,g^{-1}g_2)$$ for all $g\in\G$ and hence
 $R_A$ is independent of the first argument. 
 In consequence, also 
the symbol is independent of the first argument and we will 
write $\sigma_A(\xi)$ for it. 
In combination with Fourier inversion formula 
\eqref{eq:FI} this means that the operator
$A$ can be written as
\begin{equation}\label{EQ:inv-symbol}
  A\phi(g) = \sum_{[\xi]\in\widehat\G} 
  d_\xi \trace(\xi(g) \sigma_A(\xi) \widehat \phi(\xi) ).
\end{equation}
By this formula we can assign operators $A = \op(\sigma_A)$ to arbitrary sequences $\sigma_A\in\Sigma(\widehat G)$.
It follows\footnote{In fact, \eqref{EQ:inv-symbol0} can be 
taken as a definition
of the symbol $\sigma_A$ of $A$, from which \eqref{EQ:RS} 
and \eqref{EQ:inv-symbol} follow; see also
Section \ref{SEC:noninv}.}
that
\begin{equation}\label{EQ:inv-symbol0}
  \sigma_A(\xi) = \xi(g)^* (A\xi)(g) = (A\xi)(g)\big|_{g=1}
\end{equation}
is independent of $g$.
We refer to operators of this form as 
noncommutative Fourier multipliers.
The Plancherel identity \eqref{eq:Plancherel} 
implies that the operator $A$ is bounded 
on $L^2(\G)$ if and only if 
$\sigma_A\in\ell^\infty(\widehat\G)$, where
\begin{equation*}
   \ell^\infty(\widehat\G) =\{ \sigma_A\in \Sigma(\widehat\G) : 
   \sup_{[\xi]\in\Gh} \|\sigma_A(\xi)\|_{\rm op} < \infty \},
\end{equation*}
and $\|\cdot\|_{\rm op}$ is the operator norm
on the unitary space $\mathbb C^{d_\xi}$.
Note that there is also another version of the space $\ell^\infty(\widehat\G)$
which is realised as the weighted sequence space over Hilbert-Schmidt
norms, we refer to \cite[Section 10.3.3]{RTbook} for its properties.

We now define difference operators 
$Q\in\diff^\ell(\widehat\G)$ acting on 
sequences $\sigma \in\Sigma(\widehat\G)$ 
in terms of corresponding functions $q\in C^\infty(\G)$, which 
vanish to (at least) $\ell^{\rm th}$ order in the identity 
element $1\in\G$, and 
their interrelation with the group Fourier transform given by 
\begin{equation}
   Q\sigma  = \mathscr F \p{q(g) {\mathscr F}^{-1}\sigma}.
\end{equation}
Note, that $\sigma\in\Sigma(\widehat\G)$ implies 
${\mathscr F}^{-1}\sigma\in\mathcal D'(\G)$ and therefore the 
multiplication with a smooth function is well-defined. 
The main idea of introducing such operators is that
applying differences to symbols of Calderon--Zygmund
operators brings an improvement in the behaviour
of $\op (Q\sigma)$ since we multiply the integral kernel of
$\op (\sigma)$ by a function vanishing on its singular
set. Different collections of difference operators
have been explored in \cite{RTW10} in the 
pseudo-differential setting.

Difference operators of particular interest arise from 
matrix-coefficients of representations. 
For a fixed irreducible representation $\xi_0$ we 
define the (matrix-valued) difference 
${}_{\xi_0}\mathbb D=
({}_{\xi_0}\mathbb D_{ij})_{i,j=1,\ldots,d_{\xi_0}}$ 
corresponding to the matrix elements of 
$\xi_0(g)-\mathrm I$,  i.e. with
$$q_{ij}(g)=\xi_0(g)_{ij}-\delta_{ij},$$
$\delta_{ij}$ the Kronecker delta.
If the representation is fixed, 
we omit the index $\xi_0$. For a sequence of 
difference operators of this type,
$$\D_1={}_{\xi_1}\D_{i_1 j_1},
\D_2={}_{\xi_2}\D_{i_2 j_2}, \ldots,
\D_k={}_{\xi_k}\D_{i_k j_k},$$ 
with $[\xi_m]\in\Gh$, $1\leq i_m,j_m\leq d_{\xi_m}$,
$1\leq m\leq k$,
we define
$$\D^\alpha:=\D_1^{\alpha_1}\cdots \D_k^{\alpha_k}.$$

Among other things, it follows from 
\cite{RTW10} that 
an invariant operator $A$ belongs to the
usual H\"ormander class of pseudo-differential operators
$\Psi^0(\G)$ defined by
localisations if and only if its matrix symbol satisfies 
\begin{equation}\label{eq:RTW-inv}
   \|{\mathbb D}^{\alpha} \sigma_A(\xi) \|_{\rm op} 
   \le C_\alpha \langle\xi\rangle^{-|\alpha|}
\end{equation}
for all multi-indices $\alpha$ and for all
$[\xi]\in\Gh$. From this point of view the following
condition \eqref{eq:HM-cond} is a natural
relaxation from the $L^p$-boundedness of
zero order pseudo-differential operators to a 
multiplier theorem.

\begin{thm}\label{thm:main1} 
Denote by $\varkappa$ be the smallest even 
integer 
larger than $\frac 12\dim\G$. 
Let $A: C^\infty(\G) \to \mathcal D'(\G)$ be left-invariant.  
Assume that its symbol $\sigma_A$ 
satisfies
\begin{equation}\label{eq:HM-cond}
   \|{\mathbb D}^{\alpha} \sigma_A(\xi) \|_{\rm op} 
   \le C_\alpha \langle\xi\rangle^{-|\alpha|}
\end{equation}
for all multi-indices $\alpha$ with $|\alpha|\le \varkappa$,
and for all $[\xi]\in\Gh$. 
 Then the operator $A$ is 
of weak type $(1,1)$ and $L^p$-bounded for all $1<p<\infty$.
\end{thm}

\begin{rem}
a) The assumptions given in the theorem can be relaxed. 
For the top order difference we need only one particular difference
operator. 
Moreover, for the lower order
difference operators we only need differences associated
to the root system if $\G$ is semi-simple, and to an
extended root system for a general compact Lie group.
Such a refinement will be given in Theorem \ref{thm:main2} once we
introduced the necessary notation.
\\
b) Additional symmetry conditions for the 
operator imply simplifications. Later on we will show how the 
assumptions can be weakened for central multipliers.
\\
c) We have to round up the number of difference conditions to
even integers. This seems to be for purely technical reasons, but was
already observed similarly in \cite{W72} for central multipliers.
\\
d) The conditions are needed for the weak type $(1,1)$ property. 
Interpolation allows to reduce assumptions on the number of 
differences for $L^p$-boundedness. 
\end{rem}

Before proceeding to the proof of the theorem, we will mention some 
applications. As first example let us consider the known case of the
Riesz transform.

\begin{expl}\label{expl1}
Let us consider the partial Riesz transform 
$$\mathcal R_Z = (-\Delta)^{-1/2}\circ Z$$ associated 
to a left-invariant vector-field
$Z\in\mathfrak g$ on a Lie group $\G$. 
For simplicity we assume that $Z$ is 
normalised with respect to the Killing form on $\mathfrak g$. The 
Riesz transform is a left-invariant operator acting on $L^2(\G)$ with 
symbol 
$$
  \sigma_{\mathcal R_Z}(\xi) = (\lambda_\xi)^{-\frac12} \sigma_Z(\xi),
$$
$\sigma_Z(\xi)=(Z \xi)(1)$ the symbol 
of the left-invariant vector field, and by 
definition of the Laplacian as sum of squares we have
$$\|\sigma_{\mathcal R_Z}(\xi)\|_{\rm op}\le 1.$$ 
Note here,
that $\lambda_\xi=0$ implies that $\xi=0$ is the trivial 
representation and therefore also 
$\sigma_Z(\xi)=0$ as vector fields 
annihilate constants.
It follows from Corollary \ref{cor1} 
that this operator extends to a bounded
operator on all $L^p(\G)$, $1<p<\infty$ and is of weak type $(1,1)$,
recovering the well-known result in \cite{Stein70}.
\end{expl}

\begin{rem}
In  \cite[p. 58]{Stein70}, E.M. Stein asked whether the Riesz transform
${\mathcal R}_Z$ as well as the Riesz potentials $(-\Delta)^{i\gamma}$
($\gamma$ real) are pseudo-differential operators on $\G$. 
This is in fact true on all
closed Riemannian manifolds. 
Indeed, if $p_0$ denotes the projection to
the zero eigenspace of $-\Delta$, then we have the identity
$$(-\Delta)^z=(-\Delta+p_0)^z-p_0$$ for all complex $z$. The operator
$(-\Delta+p_0)^z$ is pseudo-differential for $\Re z<-1$ by 
\cite{Seeley} and $p_0$ is
smoothing, implying that $(-\Delta)^z$ are pseudo-differential 
of order $\Re z/2$. By calculus this extends to all $z\in\C$.
In particular, the $L^p$ boundedness in
Example \ref{expl1} also follows.
\end{rem} 

\begin{expl} Let $\rho\in[0,1]$.
We denote by $\mathscr S^0_{\rho}(\G)$ the set of all $\sigma_A\in\Sigma(\widehat G)$
satisfying symbol estimates of type $\rho$
\begin{equation*} 
    \|\mathbb D^\alpha \sigma_A(\xi)\|_{\rm op} 
    \le C_\alpha \langle\xi\rangle^{-\rho|\alpha|}
\end{equation*}
for  all multi-indices $\alpha$. Let $A=\op(\sigma_A)$ be the associated operator to such a symbol.
Then $A$ defines a bounded operator mapping
 $W^{p,r}(\G)\to L^p(\G)$ for 
 $$r\ge \varkappa (1-\rho) \left| \frac 1p-\frac12\right|,$$ 
 $\varkappa$ 
 as in Theorem~\ref{thm:main1} and $1<p<\infty$. 
See Corollary \ref{cor2}, where we give a refined
version of this.
\end{expl}

\begin{expl}\label{EX:subL}
The previous example applies in particular to the 
parametrices constructed in \cite{RTW10}. Following the notation from that paper,  we consider the
sub-Laplacian $$\mathcal L_s = \mathrm D_1^2+\mathrm D_2^2$$ on 
$\mathbb S^3$. It was shown that  it has a parametrix from 
$\op\mathscr S^{-1}_{1/2}({\mathbb S^3})$
and therefore $\mathcal L_s u \in L^p(\mathbb S^3)$ 
implies regularity for $u$. More precisely, the sub-elliptic estimate
\begin{equation}\label{eq:sub-ell-Ls}
\|u\|_{W^{p, 1-|\frac1p-\frac12|}(\mathbb S^3)}\leq 
C_p\|\mathcal L_s u\|_{L^p(\mathbb S^3)}
\end{equation}
holds true for all $1<p<\infty$.

Similarly, the ``heat'' operator
$$H = \mathrm D_3-\mathrm D_1^2-\mathrm D_2^2$$ on 
$\mathbb S^3$ has a parametrix from 
$\op\mathscr S^{-1}_{1/2}({\mathbb S^3})$.
Consequently, we also get the sub-elliptic estimate
\eqref{eq:sub-ell-Ls} with $H$ instead of $\mathcal L_s$.
\end{expl}

Similar examples can be given for arbitrary compact 
Lie groups $\G$. 
Operators in Example \ref{EX:subL} are locally hypoelliptic, but
the following corollary applies to operators which are only
globally hypoelliptic.

\begin{cor}\label{COR:vfs}
Let $X$ be a left-invariant
real vector field on $\G$. 
Then there exists a discrete exceptional set 
$\mathscr C\subset\mathrm i\mathbb R$, 
such that for any complex number 
$c\not\in\mathscr C$ the operator
$X+c$ is invertible with inverse in $\op\mathscr S^{0}_{0}(\G)$.
Consequently, the inequality
$$
  \|f\|_{L^p(\G)} \le C_p \|(X+c)f\|_{W^{p,\varkappa|\frac1p-\frac12|}(\G)}
$$
holds true for all $1<p<\infty$ and all functions $f$ from that 
Sobolev space. 
\end{cor}

We prove this corollary later, but now only give its refinement 
on $\SU2$.

\begin{expl}\label{EX:vfs}
To fix the scaling on the Lie algebra ${\mathfrak{su}(2)}$,
let $(\phi,\theta,\psi)$ be the (standard)
Euler angles on $\SU2$ and let
$D_3=\partial/\partial\psi$. Let $X$ be a left-invariant
vector field on $\SU2$ normalised so that
$\|X\|=\|D_3\|$ with respect to the Killing norm.
Then it was shown in \cite{RTW10} that
$\i \mathscr C=\frac12\Z$, and $X+c$ is invertible if and only
if $\i c\not\in \frac{1}{2}\Z$. For such $c$, the inverse 
$(X+c)^{-1}$ has symbol in $\mathscr S^{0}_{0}(\SU2)$. 
The same conclusions remain true if we replace
$\SU2$ by $\mathbb S^3$.
In particular, we get that 
$$
  \|f\|_{L^p(\mathbb S^3)} \le C_p 
  \|(X+c)f\|_{W^{p,2|\frac1p-\frac12|}(\mathbb S^3)}
$$
holds true for all $1<p<\infty$ and all functions $f$ from that 
Sobolev space.
We note that this estimate is non-localisable since operators
$X+c$ are locally non-invertible and also not 
locally sub-elliptic (unless $n=1$).
\end{expl}

\begin{rem}\label{REM:tori} 
The H\"ormander multiplier theorem \cite{Horm},
although formulated in $\Rn$, has a natural analogue on the
torus $\Tn$. The refinement in Theorem
\ref{thm:main2} on the top order difference brings a refinement
of the toroidal multiplier theorem, at least for some dimensions.
If $G=\Tn=\Rn/\Zn$, the set $\Delta_0$ in Remark \ref{REM:newDelta0}
consists of $2n$ functions $\e^{\pm 2\pi\i x_j}$, $1\leq j\leq n$.
Consequently, we have that 
$$\rho^2(x)=2n-\sum_{j=1}^n 
\p{\e^{2\pi\i x_j}+\e^{-2\pi\i x_j}}$$ 
in \eqref{eq:rhoDef}, and hence
$$\rhodiff \sigma(\xi)=2n \sigma(\xi)-\sum_{j=1}^n 
\p{\sigma(\xi+e_j)+\sigma(x-e_j)}$$ in \eqref{EQ:rhodiff},
where $\xi\in\Zn$ and $e_j$ is its $j$th unit basis vector
in $\Zn$.

A (translation) invariant operator $A$ and its symbol $\sigma_A$
are related\footnote{On the torus we can abuse the notation by
writing $\sigma_A(k)$ for $\sigma_A(e_k)$ for
$e_k(x)=\e^{2\pi\i x\cdot k}$.
For the consistent development
of the toroidal quantization of general operators on
the tori see \cite{RT-torus} or \cite{RTbook},
with an earlier partial exposition in \cite{RT-torus0}.}
 by $$\sigma_A(k)=\e^{-2\pi\i x\cdot k}
(A\e^{2\pi\i x\cdot k})=(A\e^{2\pi\i x\cdot k})|_{x=0}$$ and
$$A\phi(x)=\sum_{k\in\Zn} e^{2\pi\i x\cdot k} 
\sigma_A(k) \widehat{\phi}(k).$$
Thus, it follows from Theorem \ref{thm:main2} that, for example
on ${\mathbb T}^3$, a translation invariant operator
$A$ is weak (1,1) type and bounded on $L^p({\mathbb T}^3)$ 
for all $1<p<\infty$ provided 
that there is a constant
$C>0$ such that $$|\sigma_A(k)|\leq C,$$
$$|k||\sigma_A(k+e_j)-\sigma_A(k)|\leq C,$$ and
\begin{equation}\label{EQ:torus}
 |k|^2| \sigma_A(k)-\frac16\sum_{j=1}^3 
\p{\sigma_A(k+e_j)+\sigma_A(k-e_j)}|\leq C,
\end{equation} 
for all
$k\in\Z^3$ and all (three) unit vectors $e_j$, $j=1,2,3$.
Here in \eqref{EQ:torus}
we do not make assumptions on all second order
differences, but only on one of them.
\end{rem} 

\section{Proofs}
\label{SEC:proofs}

The proof of Theorem~\ref{thm:main1} is divided into 
several sections. First we introduce the tools 
we need to prove Calderon--Zygmund type estimates for convolution
kernels. Later on we show how to reduce the above theorem to a
statement of Coifman and de Guzman, see \cite{CdG70} and 
also \cite{CWbook}. Finally, we use properties of the
root system with finite Leibniz rules for difference
operators to prove the refinement of Theorem
\ref{thm:main1} given in Theorem \ref{thm:main2}.

\subsection{A suitable pseudo-distance on $\G$}\label{sec:2.1} 
At first we construct
a suitable pseudo-distance on the group $\G$ in terms of a 
minimal set of representations.
We now define with $n=\dim\G$
\begin{equation}\label{eq:rhoDef}
  \rho^2 (g) = n - \trace \mathrm{Ad}(g) = 
  \sum_{\xi\in\Delta_0} (d_\xi - \trace \xi(g)),
\end{equation}
where $\mathrm{Ad} : \G \to \mathrm U(\mathfrak g)
\simeq \mathrm U(\dim \G)$ 
denotes the adjoint representation of the
Lie group $\G$ and 
$$\mathrm {Ad} = 
(\dim Z(G))1 \oplus \bigoplus_{\xi\in\Delta_0} \xi$$
is its Peter-Weyl decomposition into 
irreducible components.  
Here, $1$ denotes the trivial one-dimensional representation.
 For simplicity we assume first that
the group is semi-simple, 
i.e., that the centre $Z(\G)$ of the group $\G$ is trivial. 
Later on 
we will explain the main modifications for the general situation,
see Remark \ref{REM:newDelta0}.

Note, that $\rho^2(g)$ is nonnegative by definition and smooth. 
At first we claim that $\rho$ defines a pseudo-distance 
$$d_\rho(g,h) = \rho(g^{-1}h).$$

\begin{lem}\label{lem:2.6} The above defined function $\rho(g)$ satisfies
\begin{enumerate}
\item $\rho^2(g)\ge 0$ and $\rho^2(g) = 0$ if and only if $g=1$ 
is the identity in $\G$;
\item $\rho^2$ vanishes to second order in $g=1$;
\item $\rho^2$ is a class function, in particular it satisfies 
$\rho^2(g^{-1}) ={\rho^2(g)}$ and $\rho^2(gh^{-1}) = \rho^2(h^{-1}g)$;
\item $|\rho(gh^{-1})- \rho(g)| \le C \rho(h)$ 
for some constant 
$C>0$ and all $g,h\in\G$;
\item $\rho(gh^{-1}) \le C (\rho(g)+\rho(h))$ 
for some constant $C>0$ and all $g,h\in\G$. 
\end{enumerate}
\end{lem}
\begin{proof}
(1) At first we note that for any (not necessarily irreducible) 
unitary representation $\xi$ trivially 
$|\trace\xi(g)|\le d_\xi$ 
and therefore $\Re(d_\xi - \trace\xi(g)) \ge 0$. Furthermore, 
$\trace\xi(g)=d_\xi$ is 
equivalent to $\xi(g)=\mathrm I$. Therefore, $\rho(g)=0$ 
implies that $\mathrm {Ad}(g)=\mathrm I$ 
and therefore $g\in Z(\G)$, i.e., $g=1$.

(2) Differentiating the identity $\xi(g) \xi(g)^* = \mathrm I$ 
twice at the identity element 
and denoting $\xi^*(g)=\xi(g)^*$
implies the equations
\begin{align*}
 &  \xi'(1) + {\xi^*}'(1) = 0,\\
& \xi''(1) + 2 \xi'(1)\otimes {\xi^*}'(1) + {\xi^*}''(1) = 0,
\end{align*}
the first implying that $(\Re \trace\xi)'(1)=0$, 
while the second one gives for each $v\in\mathfrak g=\mathrm T_1\G$
the quadratic form
\begin{equation*}
  (v , (\Re \trace \xi)''(1)  v) = -   \|\xi'(1) v\|_{\HS}^2.
\end{equation*}
Summing this over $\xi\in\Delta_0$ implies
\begin{equation*}
   (v, \mathrm{Hess}\, {\rho^2}(1) v) = - \sum_{\xi\in\Delta_0} 
   \|\xi'(1) v\|_{\HS}^2,
\end{equation*}
and, therefore, if $v\in\mathfrak g$ is such that the 
left-hand side vanishes, 
then $v\in\cap_{\xi}\ker\xi'(1)$. By 
$Z(\G)=\{1\}$ and the definition of 
$\rho^2(g)$ this implies $v=0$. 

(3) Obvious by construction.

(4) We observe that both the left and the right 
hand side vanish 
exactly in $h=1$ to first order. 
The existence of the constant $C$ 
follows therefore just by compactness of $\G$.

(5) follows directly by (4).
\end{proof}

\begin{rem}\label{REM:newDelta0}
If the centre of the group is non-trivial, we have to make a 
slight change to the definition of $\rho^2(g)$. 
We have to include $2\dim Z(\G)$ 
additional representations to the set $\Delta_0$ 
defined by the choice of an
isomorphism 
$Z(\G)\simeq\mathbb T^\ell=\mathbb R^\ell/{\mathbb Z^\ell}$. 
For each coordinate $\theta_j$ we include both
$\theta\mapsto \mathrm e^{\pm 2\pi\mathrm i \theta_j}$, 
suitably extended to the maximal torus and then to $\G$. 
The statement of Lemma~\ref{lem:2.6} remains true for 
both modifications. 
In the following we assume that $\Delta_0$ and $\rho(g)$ 
are defined in this way.
In general, for the statements below to be true, any extension of
$\Delta_{0}$ will work as long as the function
$\rho^{2}(g)$ in \eqref{eq:rhoDef} is the square of a distance
function on $G$ in a neighbourhood of the neutral element.
\end{rem}

\subsection{A special family of mollifiers} Let $\tilde \varphi\in C_0^\infty(\mathbb R)$ be such that $\tilde \varphi\ge 0$,
$\tilde\varphi(0)=1$ and $\tilde \varphi^{(\ell)}(0)=0$ for all $\ell\ge 1$. Then for $r>0$ we define
\begin{equation}
   \varphi_r (g) = c_r \tilde\varphi(r^{-1/n} \rho(g)), \qquad \int_\G \varphi_r(g) \d g = 1,
\end{equation}
the normalisation condition used to define $c_r$. As $r\to0$ obviously
$\varphi_r\to \delta_1\in\mathcal D'(\G)$. Let furthermore 
$$\psi_r(g) := \varphi_r(g) - \varphi_{r/2}(g).$$ 
At first we check 
the conditions of Coifman--de Guzman \cite{CdG70} (modulo the 
obvious modifications) for these functions.

\begin{lem}
\begin{enumerate}
\item $\sup_g|\varphi_r(g)| \sim c_r \sim r^{-1}$ as $r\to0$.
\item $\|\varphi_r\|_2 \sim r^{-1/2}$ as $r\to0$.
\item $\varphi_r*\varphi_s=\varphi_s*\varphi_r$.
\item $ \int_{\rho(g)\ge t^{1/n}} \varphi_r(g)\d g \le  C_N \big(\frac rt \big)^N$ for all $N\ge 0$.
\item $\int_\G |\varphi_r(gh^{-1}) -\varphi_r(g)| \d g \le C'\frac{\rho(h)} {r^{1/n}}$.
\end{enumerate}
\end{lem}
\begin{proof}
(1) We can find a chart in the neighbourhood of the identity element 
such that $\rho(g) = |x|$ and $\d g = \nu(x)\d x$ for some smooth density $\nu$ with
$\nu(0)\ne0$. Then direct calculation yields for small $r$
\begin{align*}
  c_r^{-1} &= \int_\G \tilde\varphi(r^{-1/n} |x| )\nu(x)\d x = \int_0^1 \tilde\varphi(r^{-1/n} s) s^{n-1} \int_{\mathbb S^{n-1}} \nu(s\theta) \d\theta \d s\\
  &\lesssim \int_0^1 \tilde\varphi(r^{-1/n} s) s^{n-1}  \d s \sim r.
\end{align*}

(2) follows from (1) by interpolation with the normalisation 
condition used.

(3) this follows from $\varphi_r$ being a class function.

(4) Again direct computation of the left-hand side yields for 
sufficiently small $r$
\begin{align*}
& c_r  \int_{s\ge t^{1/n}} \tilde\varphi(r^{-1/n} s) s^{n-1} \int_{\mathbb S^{n-1}} \nu(s\theta)\d\theta \d s\\
&\qquad \lesssim c_r \int_{s\ge t^{1/n}}  \tilde\varphi(r^{-1/n} s) s^{n-1} \d s 
 \sim F(\textstyle\frac tr) 
\end{align*}
for a function $F\in C_0^\infty(\mathbb R_+)$, which implies in particular the desired 
estimate.

(5) Using that $\tilde\varphi\in C^\infty_0(\mathbb R)$ the mean value theorem implies in combination with Lemma~\ref{lem:2.6}(4)
\begin{align*}
|\varphi_r(gh^{-1})-\varphi_r(g)| &= c_r |\tilde\varphi(r^{-1/n}\rho(gh^{-1}))-\tilde\varphi(r^{-1/n}\rho(g))|\\
&\lesssim c_r r^{-1/n} |\rho(gh^{-1})-\rho(g)| \lesssim c_r r^{-1/n} \rho(h).
\end{align*}
Furthermore, the first expression is non-zero for small $r$ only
if either of the terms is non-zero, which gives $\rho(g)\lesssim r^{1/n}$ or $\rho(gh^{-1})\lesssim r^{1/n}$.
This corresponds for small $r$ to two balls of radius $r^{1/n}$, 
i.e., volume $r$.
Integration over $g\in\G$ implies the desired statement.
\end{proof}

As $\psi_r$ and $\rho^n$ satisfy all assumptions of 
\cite{CdG70}, we have the 
following criterion. 

\begin{crit} Assume $A: L^2(\G)\to L^2(\G)$ is a left-invariant 
operator on $\G$ satisfying
\begin{equation}\label{eq:CdG-cond}
  \int_\G | A\psi_r(g) |^2 \rho^{n(1+\epsilon)} (g) \d g 
  \le C r^{\epsilon} 
\end{equation}
for some constants $\epsilon>0$ and $C>0$ uniform in $r$. 
Then $A$ is of weak 
type $(1,1)$ and bounded on all $L^p(\G)$ for $1<p<\infty$. 
\end{crit}

Later on we will need some more properties of the functions $\psi_r$. 
We collect them as follows

\begin{lem}\label{lem:psi-est}
Let $q\in C^\infty(\G)$ be a smooth function vanishing to order
 $t$ in $1$. Then 
\begin{equation}
   \| q(g) \psi_r(g) \|_{H^{-s}} \le C_{q,s} r^{\frac{t+s}{n} - \frac12}
\end{equation}
for all $s\in[0,1+\frac n2]$. 
\end{lem}
\begin{proof}
Note that the statement is purely local in a 
neighbourhood of $1$. 
For sufficiently small $r$ we find local co-ordinates near $1$ 
supporting $\psi_r(g)$ and satisfying $\rho(g)=|x|$. 
We write $q$ as 
Taylor polynomial $q_N(g)$ of degree $t+N$ plus remainder 
$R_N(g)=\mathcal O(\rho^{t+N+1}(g))$ and decompose 
$q(g)\psi_r(g)$ accordingly. First, we observe

\begin{align*}
  \| q_N\psi_r \|_{H^{-s}}^2 \sim&\int \langle \xi\rangle^{-2s} 
  \big| q_N(\partial_\xi) \big(\widehat {\tilde \varphi}
  (r^{1/n}|\xi|) - 
  \widehat{\tilde \varphi}(2r^{1/n}|\xi|) \big) \big|^2 \d \xi\\
  \lesssim & r^{2t/n} \int_{r^{1/n} |\xi|\le 1}   
  \langle\xi\rangle^{-2s} r^{2/n} |\xi|^2  |\xi|^{n-1} \d|\xi|
  \\& + r^{2t/n} \int_{r^{1/n}|\xi|\ge 1} r^{-M/n} 
  |\xi|^{-2s-M} |\xi|^{n-1}\d|\xi|\\
   \lesssim & r^{\frac{2t+2}n} \langle \xi\rangle^{-2s+2+n} 
   \big|_0^{r^{-1/n}} +  r^{\frac{2t-M}n} |\xi|^{-2s-M+n}
   \big|_{r^{-1/n}}^\infty \\
   \lesssim & r^{\frac{2t+2s - n}n } 
\end{align*}
where $\xi$ is (abusing notation) the 
Fourier co-variable to $x$. The 
integral is split into $r^{1/n}|\xi|\lesssim 1$ and 
$r^{1/n}|\xi|\gtrsim 1$. Second, we consider the
remainder and show that it is smaller. Indeed,
\begin{align*}
  \|R_N\psi_r\|_{H^{-s}} \lesssim   \|R_N\psi_r\|_{2} \lesssim 
 \|R_N\|_{L^\infty(\mathrm{supp}\,\psi_r)}
 \|\psi_r\|_{2} \lesssim  r^{(t+N+1)/n}r^{-1/2}
\end{align*}
and choosing $N>s-1$ the desired smallness follows. 

Assumptions we had to make were 
$-2s+2+n \ge 0$, i.e., $s\le 1 +\frac n2$ and
 $n-2s-M<0$, i.e., $M>n-2s$. 
 Furthermore, we need $s\ge 0$. The lemma is 
proven.
\end{proof}

\subsection{Difference operators and Leibniz rules} 
We recall the definition of difference operators
before Theorem \ref{thm:main1}.
For a fixed irreducible representation $\xi_0$ we 
define the (matrix-valued) difference 
$${}_{\xi_0}\mathbb D=
({}_{\xi_0}\mathbb D_{ij})_{i,j=1,\ldots,d_{\xi_0}}$$
corresponding to the symbol $\xi_0(g)-\mathrm I$. 
We denote by $\delta_{ij}$ the Kronecker delta,
$\delta_{ij}=1$ for $i=j$ 
and $\delta_{ij}=0$ for $i\not=j$. 
Thus, we have
$${}_{\xi_0}\mathbb D_{ij}=
\mathscr F \p{\xi_0(g)_{ij}-\delta_{ij}} {\mathscr F}^{-1}.$$
If the representation is fixed, 
we omit the index $\xi_0$.
As observed in \cite{RTW10} these difference 
operators satisfy the finite (two-term) Leibniz rule
\begin{equation}\label{EQ:Leibniz-first}
   \mathbb D_{ij}(\sigma\tau) = (\mathbb D_{ij}\sigma)\tau + 
   \sigma(\mathbb D_{ij}\tau) + \sum_{k=1}^{d_{\xi_0}} 
   (\mathbb D_{ik}\sigma)(\mathbb D_{kj}\tau)
\end{equation}
for all sequences $\sigma,\tau\in\Sigma(\widehat\G)$. 
Iterating this, for a composition $\D^k$ of 
$k\in\N$ difference operators of this form we have
\begin{equation}\label{EQ:Leibniz-higher}
\D^k(\sigma\tau)=
\sum_{|\gamma|,|\delta|\leq k\leq |\gamma|+|\delta|}
C_{k\gamma\delta}\ (\D^\gamma \sigma)\ 
(\D^\delta \tau),
\end{equation} 
with the summation taken over all multi-indices
$\gamma,\delta\in\N_0^{\ell^2}$, $\ell=d_{\xi_0}$, satisfying
$|\gamma|,|\delta|\leq k\leq |\gamma|+|\delta|$, 
$$\D^\gamma=\D_{11}^{\gamma_{11}}\D_{12}^{\gamma_{12}}\cdots
\D_{\ell,\ell-1}^{\gamma_{\ell,\ell-1}}
\D_{\ell\ell}^{\gamma_{\ell\ell}},$$ 
and where constants
$C_{k\gamma\delta}$ may depend on a particular form of
$\D^k$.

Denote by $\rhodiff$ the difference 
operator associated to the symbol $\rho^2(g)$ defined in 
\eqref{eq:rhoDef}, 
\begin{equation}\label{EQ:rhodiff}
\rhodiff= {\mathscr F}\ \rho^2(g)\ {\mathscr F}^{-1}.
\end{equation} 
By Lemma \ref{lem:2.6} (2), this is a second order difference operator,
$\rhodiff\in \diff^2(\Gh)$,
and in view of \eqref{eq:rhoDef}
it can be decomposed as
\begin{equation}\label{EQ:D-decomp}
  \rhodiff = - \sum_{\xi\in\Delta_0}
  \sum_{i=1}^{d_\xi} {}_\xi\mathbb D_{ii}.
\end{equation}
Therefore, after summation of the Leibniz rules 
\eqref{EQ:Leibniz-first}
we observe that
\begin{equation}
   \rhodiff(\sigma\tau) = (\rhodiff\sigma)\tau+\sigma(\rhodiff\tau)-
   \sum_{\xi\in \Delta_0} 
   \sum_{i,j =1}^{d_\xi} 
   ({}_\xi\mathbb D_{ij}\sigma)({}_\xi\mathbb D_{ji}\tau).
\end{equation}
Iterating this, we observe that
\begin{align*}
  \rhodiff^2(\sigma\tau) &= 
  \rhodiff \bigg( (\rhodiff\sigma)\tau  + 
  \sigma (\rhodiff\tau) + \sum (\mathbb D\sigma)(\mathbb D\tau) \bigg)\\
  &= (\rhodiff^2\sigma)\tau + 
  2 (\rhodiff\sigma)(\rhodiff\tau) + 
  \sigma (\rhodiff^2\tau)\\&\qquad\qquad
  + \sum \bigg( (\mathbb D\rhodiff\sigma)(\mathbb D\tau) + 
  (\mathbb D\sigma)(\mathbb D\rhodiff\tau) + 
  (\mathbb D^2\sigma)(\mathbb D^2\tau)\bigg).
\end{align*} 
In the sum the orders of difference operators 
always add up to $4$. 
Similar we obtain for higher orders $m$,
\begin{align}\label{EQ:Deltam}
  \rhodiff^m(\sigma\tau) 
  &= (\rhodiff^m\sigma)\tau +  \sigma (\rhodiff^m\tau)
  + \sum_{\ell=1}^{2m-1} \sum_j 
  ( Q_{\ell,j} \sigma )(\tilde Q_{\ell,j} \tau ) .
\end{align} 
for some difference operators 
$Q_{\ell,j}\in \diff^\ell(\widehat\G)$ and 
$\tilde Q_{\ell,j}\in\diff^{2m-\ell} (\widehat\G)$.

\subsection{Proof of Theorem~\ref{thm:main1}} 
We note that Theorem~\ref{thm:main1} 
follows from its refined version which we give as
Theorem \ref{thm:main2} below.
Let $\Delta_0$ be an extended root system
as in Remark \ref{REM:newDelta0},
and we define the family
of first order difference operators associated to $\Delta_0$ by
$$
\mathscr D^1=\{{}_{\xi_0}\mathbb D_{ij}=
\mathscr F \p{\xi_0(g)_{ij}-\delta_{ij}} {\mathscr F}^{-1}:\;
\xi_0\in\Delta_0,\; 1\leq i,j\leq d_{\xi_0} \},
$$
where $\delta_{ij}$ is the Kronecker delta.
We write $\mathscr D^k$ for the family of operators of the form
$\mathbb D^\alpha=\mathbb D_1^{\alpha_1}\cdots \mathbb D_l^{\alpha_l}$,
where $\mathbb D_1,\ldots,\mathbb D_l\in \mathscr D^1$,
and for multi-indices $\alpha=(\alpha_1,\ldots,\alpha_l)$
of any length but such that $|\alpha|\leq k$.
We note that for even $\varkappa$, 
in view of \eqref{EQ:D-decomp},
the difference operator
$\rhodiff^{\varkappa/2}$ is a linear combination of operators in
$\mathscr D^\varkappa$. In general,
clearly $\mathscr D^k\subset \diff^k(\Gh)$.

\begin{thm}\label{thm:main2} 
Denote by $\varkappa$ be the smallest even 
integer 
larger than $\frac 12\dim\G$. 
Let operator $A: C^\infty(\G) \to \mathcal D'(\G)$ be left-invariant.  
Assume that its symbol $\sigma_A$ 
satisfies
$$\|\rhodiff^{\varkappa/2}\sigma_A(\xi)\|_{\rm op}
\leq C \jp{\xi}^{-\varkappa}$$
as well as
\begin{equation}\label{eq:HM-cond2r}
   \|{\mathbb D}^{\alpha} \sigma_A(\xi) \|_{\rm op} \le C_\alpha \langle\xi\rangle^{-|\alpha|},
\end{equation}
for all operators
${\mathbb D}^{\alpha}\in \mathscr D^{\varkappa-1}$,
and for all $[\xi]\in\Gh$. 
 Then the operator $A$ is 
of weak type $(1,1)$ and $L^p$-bounded for all $1<p<\infty$.
\end{thm}

For the proof it is enough to check \eqref{eq:CdG-cond} with 
appropriate $\epsilon$. If we choose $\epsilon$ such that 
$n(1+\epsilon)=4m$, $m\in\mathbb N$, the 
condition is equivalent to 
\begin{equation}\label{eq:2.20}
   \| \rhodiff^m (\sigma_A \widehat\psi_r) 
   \|_{\ell^2(\widehat\G)} 
   \lesssim r^{\frac{2m}{n}-\frac12}.
\end{equation}
Applying the Leibniz rule \eqref{EQ:Deltam}
to the left-hand side implies for fixed 
$[\xi]\in\widehat\G$ that we have
\begin{align*}
  \| \rhodiff^m (\sigma_A \widehat\psi_r) (\xi) \|_{\HS}
  \le & \|\rhodiff^m \sigma_A(\xi)\|_{\rm op} 
  \| \widehat\psi_r (\xi) \|_{\HS} +
  \|\sigma_A(\xi)\|_{\rm op} 
  \| \rhodiff^m \widehat\psi_r (\xi) \|_{\HS} \\
  &+ \sum_{\ell=1}^{2m-1} \sum_j \| \langle\xi\rangle^\ell 
  Q_{\ell,j}\sigma_A(\xi) \|_{\rm op} \| \langle\xi\rangle^{-\ell} 
  \tilde Q_{\ell,j} \widehat\psi_r\|_{\HS}
\end{align*}
for certain differences $Q_{\ell,j}\in\diff^\ell(\widehat\G)$ 
of order $\ell$ arising from 
Leibniz rule and corresponding differences 
$\tilde Q_{\ell,j} \in \diff^{2m-\ell}(\widehat\G)$. Summing 
$d_\xi$ times these inequalities over $[\xi]\in\widehat\G$ and 
using the assumptions of Theorem \ref{thm:main2}, we can
apply 
Lemma~\ref{lem:psi-est} in the form
\begin{equation}
  \| \tilde q_{\ell,j} \psi_r \|_{H^{-\ell}} \lesssim r^{\frac {2m}n-\frac12}
\end{equation}
for $0\le \ell \le1+\frac n2$. Under the assumption 
that $2m\le 2+\frac n2$ this implies 
the desired estimate \eqref{eq:2.20}.
 
 \begin{rem}
Note, that the number of difference conditions is 
$\varkappa=2m$, where
$\frac n2 < \varkappa \le 2+\frac n2$, 
as we have to assure that 
$\epsilon>0$ and that 
Lemma~\ref{lem:psi-est} is applicable.
 \end{rem}

 \section{Applications to central multipliers}
 \label{SEC:central}

We turn to some applications of Theorem~\ref{thm:main1}. 
First we 
collect some statements about central sequences 
$\sigma\in\Sigma(\widehat\G)$, 
$\sigma(\xi) = \sigma_\xi \mathrm I$. 
Particular examples of interest are defined in terms of 
$d_\xi$ or 
$\lambda_\xi$ or 
appear in connection with invariant multipliers on homogeneous spaces
with respect to massive subgroups. For the sake of simplicity we assume
in the sequel that $\sigma_\xi$ is defined on the full weight lattice $\Lambda\subset\mathfrak t^*$ 
for the Cartan subalgebra $\mathfrak t$,
and treat $\widehat G$ as subset of $\Lambda$, representations identified with their dominant highest weights.
We refer to e.g.
\cite{Fegan} for Weyl group, Weyl dimension and
Weyl character formula. We will use a notion of difference 
operators on the weight lattice;
difference operators of higher order are understood as 
iterates of first order forward 
differences on this lattice. 

 \subsection{Some auxiliary statements on central sequences}
 \label{sec:3.1} 
 First, 
we consider the sequence $d_\xi$ of dimensions of representations. We extend the sequence $d_\xi$ to the full weight lattice by Weyl's dimension formula (after fixing the set $\Delta_0^+$ of positive roots). 
\begin{lem}\label{lem:3.2}
The sequence $d_\xi$ satisfies the polynomial bound
\begin{equation}\label{eq:3.1}
  d_\xi \lesssim  \langle\xi\rangle^{\ell},\qquad \ell = 
  |\Delta_0^+| \le \textstyle\frac12(n - \mathrm{rank}\,\G),
\end{equation}
together with the hypoellipticity estimate 
\begin{equation}\label{eq:3.2}
 \frac{ |\triangle_k d_\xi| }{|d_\xi|}\le C_k   \langle\xi\rangle^{-k} ,\qquad d_\xi\ne0,
\end{equation}
for any difference operator $\triangle_k$ of order $k$ acting on the 
weight lattice.   
\end{lem}
\begin{proof}
We recall first, that the dimension $d_\xi$ can be expressed in 
terms of the heighest weights (for simplicity also denoted by the 
variable $\xi\in\Lambda \subset \mathfrak t^*$, 
$\mathfrak t = \mathrm T_1 \mathcal T$ for 
$\mathcal T\subset\G$ a maximal torus of $\G$) by Weyl's 
dimension formula
\begin{equation}
   d_\xi = \frac{\prod_{\alpha\in\Delta_0^+}(\xi+\rho,\alpha)}
   {\prod_{\alpha\in\Delta_0^+}(\rho,\alpha)},\qquad \rho = 
   \frac12 \sum_{\alpha\in\Delta_0^+}\alpha.
\end{equation}  
The sum goes over the positive roots $\alpha\in\Delta_0^+$, 
which form a subset 
of the set $\Delta_0$ used before. Weyl's dimension formula directly implies
\eqref{eq:3.1} from $\langle\xi\rangle \sim 1+||\xi||$.

In order to prove \eqref{eq:3.2} we consider first an arbitrary
difference of first order of the form 
$\triangle_\tau d_\xi = d_{\xi+\tau}-d_\xi$ for a suitable
lattice vector $\tau\in\Lambda$. Then, an elementary calculation shows that
\begin{align*}
  \frac{\triangle_\tau d_\xi}{d_\xi} = \frac{\prod_{\alpha\in\Delta_0^+}
  \big((\xi+\rho,\alpha)+(\tau,\alpha)\big)-\prod_{\alpha\in\Delta_0^+}(\xi+\rho,\alpha)}
  {\prod_{\alpha\in\Delta_0^+} (\xi+\rho,\alpha)} 
\end{align*}
and, therefore, we see that the right-hand side indeed 
behaves like $\langle\xi\rangle^{-1}$ for all $d_\xi\ne0$. 
The full statement follows in analogy. 
\end{proof}

In the following we extend the family of characters $\chi_\xi$ from $[\xi]\in\widehat G$ 
to the full weight lattice using the Weyl character formula
\begin{equation}
   j(\exp x) \chi_\xi(\exp x) = \sum_{\omega\in\mathcal W} \sign(\omega) \e^{2\pi\i (\omega(\xi+\rho),x) } 
\end{equation}
for $x\in\mathfrak t\subset\mathfrak g$ the Cartan subalgebra and $\mathcal W$ its Weyl group. As usual 
\begin{equation}
   j(\exp x) = \sum_{\omega\in\mathcal W} \sign(\omega)\mathrm e^{2\pi\mathrm i(\omega\rho,x)} 
\end{equation}
denotes the Weyl denominator.  As $\chi_\xi : \mathcal T\to \C$ is invariant under the adjoint action it extends to a unique central function on the group $G$. We collect two properties of these functions related to averaging over orbits of the Weyl group.

\begin{lem}
Let $\mathcal O_\xi :=  \{ \omega\xi : \xi\in\mathcal W\}$. Then
\begin{equation}\label{eq:4.6}
    \sum_{\xi'\in\mathcal O_\xi} \chi_{\xi'}(\exp x) = \sum_{\xi'\in\mathcal O_\xi} \e^{2\pi \i (\xi',x)}
\end{equation}
and 
\begin{equation}\label{eq:4.7}
  \sum_{\xi'\in\mathcal O_\xi} \chi_{\xi'}(\exp x) \chi_{\xi_*}(\exp x) = \sum_{\xi'\in\mathcal O_\xi} 
  \chi_{\xi_*+\xi'}(\exp x)
\end{equation}
for any fixed pair $\xi, \xi_*\in\Lambda$. Furthermore,
\begin{equation}\label{eq:4.7b}
   \int_G \chi_{\xi_*}(g) \overline{ \chi_\xi(g)} \d g = \begin{cases} \sign(\omega) & \exists \omega\in\mathcal W : \omega (\xi+\rho) = \xi^*+\rho, \\ 0 &\text{otherwise.}\end{cases} 
\end{equation}
\end{lem}
\begin{proof}
Using Weyl character formula we obtain
\begin{align*}
   j(\exp x)    \sum_{\xi'\in\mathcal O_\xi} \chi_{\xi'}(\exp x) &= \sum_{\xi'\in\mathcal O_\xi} \sum_{\omega\in\mathcal W} \sign(\omega) \e^{2\pi\i (\omega(\xi' + \rho),x)} \\
 &  = \sum_{\omega\in\mathcal W} \sign(\omega) {\bigg( \sum_{\xi'\in\mathcal O_\xi} \e^{2\pi \i(\omega\xi',x)}\bigg)} \e^{2\pi\i(\omega \rho,x)}\\
 &=j(\exp x)  \sum_{\xi'\in\mathcal O_\xi} \e^{2\pi \i (\xi',x)}
\end{align*}
using that elements of $\mathcal W$ permute the orbit $\mathcal O_\xi$ and hence the first identity. Similarly we obtain
\begin{align*}
 (j(\exp x))^2&\sum_{\xi'\in\mathcal O_\xi }  \chi_{\xi_*}(\exp x) \chi_{\xi'}(\exp x) \\
&=
 \sum_{\xi'\in\mathcal O_\xi } \bigg(\sum_{\omega\in\mathcal W}\sign(\omega) 
   \mathrm e^{2\pi\mathrm i(\omega (\xi_*+\rho),x)}\bigg) 
 \bigg(\sum_{\omega'\in\mathcal W}\sign(\omega')
   \mathrm e^{2\pi\mathrm i(\omega' (\xi'+\rho),x)}\bigg) \\
&=\sum_{\omega\in\mathcal W}\sum_{\omega'\in\mathcal W} \bigg(\sum_{\xi'\in\mathcal O_\xi } 
\sign(\omega)
\mathrm e^{2\pi\mathrm i(\omega(\omega^{-1}\omega'\xi'+\xi_*+\rho),x)} \bigg)
\sign(\omega') \mathrm e^{2\pi\mathrm i(\omega'\rho,x)}\\
&= (j(\exp x))^2 \sum_{\xi'\in\mathcal O_\xi} 
\chi_{\xi_*+\xi'}(\exp x).
\end{align*}
Furthermore, \eqref{eq:4.7b} follows by Weyl integration formula,
\begin{equation*}
   \int \chi_{\xi_*}(g) \overline{\chi_{\xi}(g)}\d g = \frac1{|\mathcal W|} \sum_{\omega,\omega'\in\mathcal W} 
   \sign(\omega\omega')  
    \int_{\R^k/\mathbb Z^k}  
   \e^{2\pi\i (\omega(\xi_*+\rho)-\omega'(\xi+\rho),x)} \d x
\end{equation*}
combined with the orthogonality relations of trigonometric functions and the fact that $\mathcal W$ acts simply and transitively on the chambers.
\end{proof}

\begin{lem} Assume $G$ is semi-simple. Then by \eqref{eq:4.6}
\begin{equation}\label{eq:4.8}
    \trace \mathrm{Ad}(\exp x) - \rank G = \sum_{\xi\in\Delta_0}   \sum_{\xi'\in\mathcal O_\xi} \e^{2\pi \i (\xi',x)} = \sum_{\xi\in\Delta_0} \sum_{\xi'\in\mathcal O_\xi} \chi_{\xi'}(\exp x). 
\end{equation}
\end{lem}

For the following we 
assume that $\sigma\in\Sigma(\widehat\G)$ is central, $\sigma(\xi)=\sigma_\xi\mathrm I$. 
This corresponds to a distribution
${\mathscr F}^{-1} \sigma\in\mathcal D'(\G)$ invariant 
under the adjoint action of the group. The 
following lemma explains the action of the difference operator 
$\rhodiff$ on
$\sigma$. We understand $d_\xi\sigma_\xi$ as scalar sequence on the lattice
of dominant weights extended by the action of the Weyl group
\begin{equation*}
    \sigma_{\xi'} = \sigma_\xi,\qquad \text{if}\quad\exists\omega\in\mathcal W : \xi'+\rho = \omega(\xi+\rho)
\end{equation*}
and recall that $d_\xi$ and $\chi_\xi$ behave odd
\begin{equation*}
    d_{\xi'} =\sign(\omega) d_\xi,\qquad 
    \chi_{\xi'} =\sign(\omega) \chi_\xi,\qquad  \text{if}\quad\exists\omega\in\mathcal W : \xi'+\rho = \omega(\xi+\rho).
\end{equation*}

\begin{lem}\label{lem:centr-diff} 
Assume $G$ is semi-simple.
Then there exists a second order difference operator $\triangle_2$ 
acting on the lattice 
of heighest weights such that
\begin{equation}\label{eq:3.4}
   d_\xi \rhodiff \sigma = \triangle_2(d_\xi\sigma_\xi)\mathrm I
\end{equation}
holds true.
\end{lem} 

\begin{proof}
It suffices to prove the formula for elementary sequences $\sigma_\xi$ which are $1/d_{\xi_*}$ 
for some dominant $\xi=\xi_*$ and $0$ otherwise. Then $\rhodiff\sigma$ is the Fourier transform of
$\rho^2(g)\chi_{\xi_*}(g)$, which in turn can be calculated based on equation \eqref{eq:4.7} and \eqref{eq:4.8},
\begin{align*}
\mathscr F[\rhodiff\sigma]& = (\dim G - \trace\mathrm{Ad} ) \chi_{\xi_*}
 = \sum_{\xi\in\Delta_0} \bigg(\delta_{\xi}\chi_{\xi_*} - \sum_{\xi'\in\mathcal O_{\xi}}
  \chi_{\xi'}\chi_{\xi_*} \bigg)\\
&=\sum_{\xi\in\Delta_0} 
\bigg(\delta_{\xi}\chi_{\xi_*} -  \sum_{\xi'\in\mathcal O_{\xi}} \chi_{\xi_*+\xi'}\bigg)
=
\mathscr F[(d_\xi^{-1}\triangle_2(d_\xi\sigma_\xi))\mathrm I]  
\end{align*}
with $\delta_\xi=|\mathcal O_\xi|$ and the difference operator
\begin{equation*}
  \triangle_2 \tau_\xi = \sum_{\xi'\in\Delta_0}\bigg( \delta_{\xi'}\tau_\xi -\sum_{\xi''\in\mathcal O_{\xi'}} \tau_{\xi-\xi''}\bigg)
\end{equation*}
acting on the weight lattice $\Lambda$. Near the walls of the Weyl chamber we made use of the particular extension of $\sigma_\xi$.
The difference operator $\triangle_2$ annihilates linear 
functions on the lattice
and is therefore of second order.
\end{proof}

\begin{expl}
On the group $\mathbb S^3\simeq\mathrm{SU}(2)$ we obtain for 
$\rhodiff \sim d_1-\trace t^{1}$ (in the notation 
of \cite{RTbook}) that central sequences $\sigma^\ell$ satisfy
\eqref{eq:3.4} with 
$\triangle_2\sigma^\ell =2 \sigma^\ell- \sigma^{\ell-1} - \sigma^{\ell+1}$,
 which is (up to sign) the usual second order 
difference on $\frac12\mathbb Z$.
\end{expl}

\begin{rem}
The statement of Lemma~\ref{lem:centr-diff} extends to arbitrary compact groups. The additional representations used to define $\rho^2(g)$ give more summands adding up to another second order difference operator on the lattice.
\end{rem}

\begin{rem}
N.~Weiss used in \cite{W72} the remarkably similar looking 
function
\begin{equation*}
\gamma(\exp\tau) = 
\sum_{\omega\in\mathcal W} 
\mathrm e^{2\pi\mathrm i (\omega\rho,\tau)} 
- |\mathcal W|,
\end{equation*}
$\mathcal W$ the Weyl group and again $\rho$ 
the Weyl vector, in 
place of our distance function 
$\rho^2(g) = \dim\G - \trace {\mathrm{Ad}(g)}$. This function
seems to simplify the treatment of central multipliers (as the 
associated difference operator $\delta$ acts in a much simpler 
way on central sequences), but it does not allow the use of a 
finite Leibniz rule which is important for our proof in the
non-central case. 
It is remarkable that $\delta d_\xi=0$.
\end{rem} 

\subsection{Functions of the Laplacian} 
We say a bounded
function $f$ defined on a normed linear space
$V$ has an asymptotic expansion at $\infty$, 
\begin{equation}\label{eq:def-hom}
  f(\eta) \sim \sum_{k=0}^\infty f_k(\eta),\qquad |\eta|\to\infty
\end{equation}
if  there exist functions $f_k(\eta)$, 
homogeneous of order $k$ for large $\eta$, such that 
\begin{equation}
  |f(\eta)-\sum_{k=0}^N f_k(\eta)|\le C_N (1+|\eta|)^{-N}
\end{equation}
holds true for certain constants $C_N$. 
We fix a maximal torus $\mathcal T$ of $\G$ and denote by $\mathfrak t^*$ the dual of its Lie algebra. 

\begin{lem}
Assume $f:\mathfrak t^*\to\mathbb C$ is bounded, even under the action of the Weyl group,
$$ 
  f(\xi)=f(\xi') \qquad\text{if $\xi'+\rho=\omega(\xi+\rho)$ for some $\omega\in\mathcal W$},
$$
and has an 
asymptotic expansion into smooth 
 components
at $\infty$ and denote $f(\xi)$ its restriction to 
the weight lattice $\Lambda\subset\mathfrak t^*$.
Then the central sequence $f(\xi)\mathrm I$ defines an 
$L^p$-bounded multiplier on $\G$ for all $1<p<\infty$.
\end{lem}

\begin{rem}\label{REM:finite-kappa}
It is enough to assume the asymptotic expansion up to 
fixed finite order $\varkappa$ as in Theorem \ref{thm:main1}.
\end{rem}

\begin{proof}
We identify $\mathfrak t^*$ with $\mathbb R^t$, $t=\rank\G$, 
which is the space $V$ in definition
\eqref{eq:def-hom}. 

In a first step let $f_k(\eta)$ be smooth and homogoneous of degree
$-k$ on $|\eta|\ge1$. Then $f_k\in S^{-k}(\mathbb R^t)$ 
and by the arguments of 
\cite[Theorem 4.5.3]{RTbook} we immediately get
that the restriction of $f$ to the lattice belongs to the symbol class
$\mathscr S^{-k}_1({\mathcal T})$. 

Furthermore, lattice
differences preserve $\mathcal O\big((1+|\eta|)^{-N}\big)$ for any $N$.
Therefore, choosing $N$ in dependence on the order of the difference
we immediately see that the restriction of $f$ to the lattice
belongs to $\mathscr S^{0}_1({\mathcal T})$.

In order to obtain the $L^p$-boundedness we follow the proof of 
Theorem~\ref{thm:main1}. Note that $\psi_r$ is defined in terms of
the pseudo-distance $\rho$ and therefore central. Hence
$\widehat\psi_r(\xi)$ is a central sequence (also denoted by $\widehat\psi_r(\xi)$
for the moment and extended evenly to the full lattice) and thus by 
Lemma~\ref{lem:centr-diff} in combination 
with Lemma~\ref{lem:3.2} we obtain
the desired bounds for the HS-norm of
$$ \rhodiff (f(\xi) \widehat\psi_r(\xi))=
\frac1{d_\xi} \triangle_2(d_\xi f(\xi)\widehat\psi_r(\xi))$$ 
and for corresponding higher differences with respect to $\rhodiff$.
\end{proof}

\begin{cor}\label{cor1}
Assume $f:\mathbb R_+\to\infty$ has an asymptotic 
expansion up to order $\varkappa$ into homogeneous 
components at $\infty$. Then $f(-\Delta)$ is bounded
on $L^p(\G)$ for $1<p<\infty$.
\end{cor}
\begin{proof}
This follows from the fact that 
$$\lambda_\xi^2 = ||\xi+\rho||^2-||\rho||^2$$ 
is even and has the desired asymptotic expansion in $\xi$. This implies that 
$f(\lambda_\xi^2)$ also has an asymptotic
expansion, see Remark \ref{REM:finite-kappa},
 and one-dimensionality allows one to 
choose the components of the expansion as smooth functions.
\end{proof}

\begin{rem}  
Coifman and G.~Weiss showed in \cite{CW73} that central
multipliers correspond to $L^p(\G)$-bounded operators if 
$\mathscr D(d_\xi\sigma_\xi)$ is an
$L^p(\mathcal T)$-bounded multiplier on the 
corresponding lattice, where $\mathscr D$ is the 
product of elementary (backward) differences
 $\triangle_{-\alpha}$ corresponding to 
the positive roots $\alpha\in\Delta_0^+$. 


\end{rem} 

\section{Applications to non-central operators}
\label{SEC:general}

In this section we give applications to invariant and
non-invariant operators. Difference operators 
$\D^\alpha$ in this section correspond to those in
Theorem \ref{thm:main1} for simplicity of the
formulations. However, in Remark \ref{REM:mods}
we explain that those associated to the extended
root system analogously to those in Theorem \ref{thm:main2}
will suffice.

\subsection{Mapping properties of operators of order zero.} 
As a second main example we 
consider operators associated to symbols 
$\mathscr S^0_\rho(\G)$ of type $\rho\in[0,1]$, i.e.
matrix symbols for which 
\begin{equation*} 
    \|\mathbb D^\alpha \sigma_A(\xi)\|_{\rm op} 
    \le C_\alpha \langle\xi\rangle^{-\rho|\alpha|},
\end{equation*}
holds for all $\alpha$ and all $[\xi]\in\Gh$, 
and ask for mapping properties of such 
operators within Sobolev spaces over $L^p(\G)$.
Such symbol classes appear naturally as parametrices
for non-elliptic operators, see Example \ref{EX:subL}
and Corollary \ref{COR:vfs}. We now give a refined
version of a multiplier theorem for such operators:

\begin{cor}\label{cor2} Let $\rho\in[0,1]$ and let
$\varkappa$ be the smallest even integer larger 
than $\frac12 \dim \G$.
Assume that $A$ is a left-invariant operator on
 $\G$ with matrix symbol $\sigma_A$ 
 satisfying
\begin{equation}\label{EQ:rhos}
    \|\mathbb D^\alpha \sigma_A(\xi)\|_{\rm op} 
    \le C_\alpha \langle\xi\rangle^{-\rho|\alpha|}
    \;\textrm{ for all }\; |\alpha|\leq\varkappa
\end{equation}
and all $[\xi]\in\Gh$. 
Then $A$ is a bounded 
operator mapping the Sobolev space
$W^{p,r}(\G)$ into $L^p(\G)$ for $1<p<\infty$ 
and $$r= \varkappa(1-\rho)\left|\frac1p-\frac12\right|.$$
\end{cor}

\begin{proof}
The proof follows  by interpolation from two 
end point statements, the trivial one for $p=2$ and
the fact that 
$\langle\xi\rangle^{-\varkappa(1-\rho)}\sigma_A(\xi)$ 
defines an operator of weak type $(1,1)$ on $L^1(\G)$.
The latter follows from Theorem~\ref{thm:main1} 
in combination with Leibniz rule \eqref{EQ:Leibniz-higher} 
for difference 
operators,
\begin{equation*}
 \| \mathbb D^\alpha \langle\xi\rangle^{-\varkappa(1-\rho)}\sigma_A(\xi)\|_{\rm op}
 \lesssim \sum_{\ell,m\le|\alpha|\le \ell+m}   
 \langle\xi\rangle^{-\varkappa(1-\rho) - \ell -m \rho}
 \lesssim 
 \langle\xi\rangle^{-\varkappa+(\varkappa-|\alpha|)\rho }
\end{equation*}
which can be estimated by $\langle\xi\rangle^{-|\alpha|}$ 
whenever $|\alpha|\le \varkappa$.
\end{proof}
Similar to Remark \ref{REM:mods}, 
Corollary \ref{cor2} remains true if in
\eqref{EQ:rhos} we take only the single difference
$\rhodiff$ of order $\varkappa$ and only
those differences that are
associated to the extended root system $\Delta_0$ for
$|\alpha|\leq\varkappa-1$, if we apply
Theorem \ref{thm:main2} instead of
Theorem \ref{thm:main1} in the proof.

We also note that the variable coefficient version 
$\mathscr S^m_{\rho,\delta}(\G)$ of these classes
$\mathscr S^m_{\rho}(\G)$, especially the class
$\mathscr S^m_{1,\frac12}(\G)$, played an important role
in the proof of the sharp G{\aa}rding inequality on
compact Lie groups in \cite{RT-Garding}.

\subsection{Proof of Corollary \ref{COR:vfs}} 

Let $X$ be left-invariant vector field on the group $G$ with 
$\sigma_X(\xi) = (X\xi)(1)$ as its symbol. 
We assume\footnote{This can always be arranged by 
diagonalising symmetric matrices; see also
\cite[Remark 10.4.20]{RTbook}.}
that the bases of the representation spaces are 
chosen such that $\sigma_X(\xi)$ is diagonal for all 
$[\xi]\in\widehat G$.
Let further $[\eta]\in\widehat G$ be a fixed representation 
with associated differences 
$\mathbb D_{ij}={}_{\eta}\mathbb D_{ij}$. 
Then for some
$\tau_{ij}$ we have
$$
   \mathbb D_{ij} \sigma_X = (X\eta_{ij})(1) 
   I_{d_\xi\times d_\xi} = \tau_{ij} I_{d_\xi\times d_\xi}
$$ 
as can be seen immediately on the Fourier side and follows from 
$\mathscr F\delta_1 = I_{d_\xi\times d_\xi}$. 
By our choice of representation spaces, $\tau_{ij}=0$ for 
$i\ne j$ and $\sum_j \tau_{jj}=0$. 
The latter one is just another formulation of the fact 
that the derivatives of the character $\chi_\eta(x) = \trace \eta(x)$ 
vanish in the identity element $1$.
Now $$\sigma_{X+c}(\xi)=\sigma_X(\xi)+cI$$
is invertible for all $\xi$, 
whenever $c\not\in\mathrm{spec}(-X)\subset\mathrm i\mathbb R$. 
For such $c$ the Leibniz rule 
\eqref{EQ:Leibniz-first}
for $\mathbb D_{ij}$ implies
$$
  0 = (\mathbb D_{ij} \sigma_{X+c}^{-1}) 
  \sigma_{X+c} + \tau_{ij} \sigma_{X+c}^{-1} + 
  \sum_{k=1}^{d_\eta} \tau_{kj} 
\mathbb (\D_{ik} \sigma_{X+c}^{-1}),
$$
so that $$\mathbb D_{ij} \sigma_{X+c}^{-1}=0$$ for $i\ne j$
and $c+\tau_{jj}\not\in \mathrm{spec}(-X)$, and
$$
 \mathbb D_{jj} \sigma_{X+c}^{-1} = 
 - \tau_{jj} \sigma_{X+c}^{-1} (\sigma_{X+c+\tau_{jj}})^{-1}
= -\tau_{jj} (\sigma_{X}+cI)^{-1} 
(\sigma_{X}+(c+\tau_{jj})I)^{-1}.
$$
Using this recursion formula we see that 
$\sigma_{X+c}^{-1} \in \mathscr S^0_{0}(\G)$ 
provided all appearing
matrix inverses exist, which means 
$$c\not\in \mathrm{spec}(-X) -\mathrm i  
\mathbb N[\tau_{11},\ldots ,\tau_{ll}],$$
where the latter stands for 
the set of all linear combinations of $\tau_{11},\ldots,\tau_{ll}$
with integer coefficients.
Outside this exceptional set of parameters  
by Corollary \ref{cor2} we conclude the $L^p$-estimate
$$
  \| f\|_{L^p(\G)} \le C_p 
  \| (X+c) f \|_{W^{p,\varkappa|\frac{1}{p}-\frac12|}(\G)}
$$
for all $1<p<\infty$.

\subsection{Non-invariant pseudo-differential operators}
\label{SEC:noninv}

The result for multipliers implies the 
$L^p$-boundedness for non-invariant
operators if we assume sufficient regularity of the symbol.
Again, such a result is an extension of the 
$L^p$-boundedness of pseudo-differential operators.

Let $A:C^\infty(\G)\to {\mathcal D}'(\G)$ be a linear
continuous operator (not necessarily invariant).
Following \cite{RTbook}, we define its matrix symbol
$$\sigma_A:\G\times\Gh\to\bigcup_{[\xi]\in\Gh}
\C^{d_\xi\times d_\xi}$$ so that
for each $(x,[\xi])\in\G\times\Gh$, the matrix
$\sigma_A(x,\xi)\in \C^{d_\xi\times d_\xi}$ is given by
$$\sigma_A(x,\xi)=\xi(x)^* (A\xi)(x).$$ 
In particular, for the left-invariant operators we have
\eqref{EQ:inv-symbol0}.
Consequently,
it was shown in \cite{RTbook} that such symbols are 
well-defined on $\G\times\Gh$
and that the operator $A$ can be quantised as
$$ A\phi(x) = \sum_{[\xi]\in\widehat\G} 
  d_\xi \trace(\xi(x)\sigma_A(x,\xi) \widehat \phi(\xi)).
$$
We also have the relation \eqref{EQ:RS} in this setting.

Let
$\partial_{x_j}$, $1\le j\le n$, be a 
collection of left invariant first order differential
operators corresponding to some linearly independent family
of the left-invariant vector fields on $G$. We denote
$\partial_x^{\beta}:=\partial_{x_1}^{\beta_1}\cdots
\partial_{x_n}^{\beta_n}$.
In \cite{RTbook}, and completed in \cite{RTW10}, 
it was shown that the H\"ormander class
$\Psi^m(\G)$ of pseudo-differential operators on $\G$
defined by localisations can be characterised in terms
of the matrix symbols. In particular, we have
$A\in\Psi^m(\G)$ if and only if
its matrix symbol $\sigma_A$ 
satisfies
$$
   \|\partial_x^\beta{\mathbb D}^{\alpha} 
   \sigma_A(x,\xi) \|_{\rm op} \le 
   C_{\alpha,\beta} \langle\xi\rangle^{m-|\alpha|}
$$
for all multi-indices $\alpha,\beta$, 
for all $x\in\G$ and $[\xi]\in\Gh$. 
For the $L^p$-boundedness it is sufficient to impose 
such conditions up to finite orders as follows,
extending Theorem \ref{thm:main1} to the
non-invariant case: 

\begin{thm}\label{thm:main3} 
Denote by $\varkappa$ be the smallest even 
integer  
larger than $\frac n2$, 
$n$ the dimension of the group $\G$. 
Let $1<p<\infty$ and let $l>\frac{n}{p}$ be an integer.
Let $A: C^\infty(\G) \to \mathcal D'(\G)$ 
be a linear continuous operator
such that its matrix symbol $\sigma_A$ 
satisfies
\begin{equation}\label{eq:HM-cond-p}
   \|\partial_x^\beta{\mathbb D}^{\alpha} 
   \sigma_A(x,\xi) \|_{\rm op} \le 
   C_{\alpha,\beta} \langle\xi\rangle^{-|\alpha|}
\end{equation}
for all multi-indices $\alpha,\beta$ with 
$|\alpha|\le \varkappa$ and 
$|\beta|\leq l$, for all $x\in\G$ and $[\xi]\in\Gh$.  
Then the operator $A$ is 
bounded on $L^p(\G)$.
\end{thm}

\begin{rem}\label{REM:mods}
The modifications similar to other formulations of multiplier
theorems regarding the choice of difference operators remain true
in a straightforward way. For example, it is enough to
impose difference conditions
$\D^\alpha$ in \eqref{eq:HM-cond-p} only with respect to
the (extended) root system. Thus, 
in analogy with Theorem \ref{thm:main2}, 
the conclusion of Theorem \ref{thm:main3} 
remains true if we 
impose 
$$\|\partial_x^\beta\rhodiff^{\varkappa/2}\sigma_A(x,\xi)\|_{\rm op}
\leq C \jp{\xi}^{-\varkappa}$$ as well as
\eqref{eq:HM-cond-p} only for 
$\D^\alpha\in {\mathscr D}^{\varkappa-1}$, for all
$|\beta|\leq l$. Similarly, Corollary \ref{cor2}
can be extended to the general (non-invariant) case.
\end{rem} 

\begin{proof}
Let $Af(x)=(f*r_A(x))(x)$, where $$r_A(x)(y)=R_A(x,y)$$
denotes the right-convolution kernel of $A$.
Let $$A_y f(x):=(f*r_A(y))(x),$$
so that $A_x f(x)=Af(x)$.
Then
$$
  \| Af \|_{L^p(\G)}^p  = 
        \int_\G |A_x f(x)|^p\ {\rm d}x \\
  \leq  \int_\G \sup_{y\in \G} |A_y f(x)|^p\ {\rm d}x.
$$
By an application of the Sobolev embedding theorem we get
$$
  \sup_{y\in \G} |A_y f(x)|^p \leq
  C \sum_{|\alpha|\leq l} \int_\G 
  |\partial_y^\alpha A_y f(x)|^p
        \ {\rm d}y.
$$
Therefore, using the
Fubini theorem to change the order of integration, we obtain
\begin{eqnarray*}
    \| Af \|_{L^p(G)}^p
  &\leq & C \sum_{|\alpha|\leq l} \int_\G \int_\G
        | \partial_y^\alpha A_y f(x) |^p
        \ {\rm d}x\ {\rm d}y \\
  &\leq & C \sum_{|\alpha|\leq l} \sup_{y\in \G} \int_\G
        | \partial_y^\alpha A_y f(x) |^p\ {\rm d}x \\
  & = & C \sum_{|\alpha|\leq l} \sup_{y\in \G}
        \| \partial_y^\alpha A_y f \|_{L^p(\G)}^p \\
  &\leq & C \sum_{|\alpha|\leq l} \sup_{y\in \G}
        \| f\mapsto f*\partial_y^\alpha r_A(y)\|_
        {{\mathcal L}(L^p(\G))}^p
        \|f\|_{L^p(\G)}^p \\
  &\leq & C 
          \|f\|_{L^p(\G)}^p,
\end{eqnarray*}
where the last inequality holds due to 
Theorem \ref{thm:main1}.
\end{proof}

\end{document}